\documentclass[12pt,a4paper]{article}
\usepackage{bbm}
\usepackage{mathrsfs}
\usepackage[latin1]{inputenc}
\usepackage{amsmath}
\usepackage{amsfonts}
\usepackage{amssymb}
\usepackage{amsthm}
\usepackage{psfrag}
\usepackage{indentfirst}
\usepackage{graphicx}
\newcommand{\lel}{\left\langle}
\newcommand{\rir}{\right\rangle}
\newcommand{\diag}{\text{diag}}
\newtheorem{theorem}{Theorem}[section]
\newtheorem{definition}[theorem]{Definition}
\newtheorem{lemma}[theorem]{Lemma}
\newtheorem{proposition}[theorem]{Proposition}

\newtheorem{ass}[theorem]{Assumption}
\title{Backward Stochastic Differential Equations with Continuous Coefficients in a Markov Chain Model and with Applications to European Options}
\author{Dimbinirina Ramarimbahoaka \thanks{Department of Mathematics and Statistics, University of Calgary, 2500 University Drive NW, Calgary, AB, T2N 1N4, Canada.} \and Zhe Yang \footnotemark[1] \and Robert J. Elliott \thanks{Haskayne School of Business, University of Calgary, 2500 University Drive NW, Calgary, AB, T2N 1N4, Canada.}  \thanks{School of Mathematical Sciences, University of Adelaide, SA 5005, Australia.}}
\date{}
\begin{document}
\maketitle
\begin{abstract}
In this paper we discuss backward stochastic differential equations with Markov chain noise, having continuous drivers. We obtain the existence of a solution which is possibly not unique. Moreover, we show there is a minimal solution for this kind of equation and derive the corresponding comparison result. This is applied to pricing of European options in a market with Markov chain noise.
\end{abstract}
\section{Introduction}

Backward stochastic differential equations (BSDEs) have been used as pricing and hedging tools in Finance. Applications of BSDEs in Finance are usually focused on a market where prices follow geometric Brownian motion or other related diffusion process models. Hence the BSDEs in such cases are driven by Brownian motions. We particularly mention the works of El Karoui and Quenez \cite{KQ} and \cite{KQ2}. \\
\indent The first work of Pardoux and Peng \cite{ParPeng1} on general BSDEs, considers equations of the form:
$$
Y_t = \xi + \int_t^T g(s, Y_s, Z_s) ds -\int_t^T  Z_sdB_s
,~~~~~t\in[0,T],
$$
where $B$ is a Brownian Motion, $g$ is the driver, or drift coefficient, and $g(t,y,z)$ is Lipschitz continuous in the variables $y$ and $z$. In this case, the solution of the BSDE is unique. In derivative pricing and hedging, this leads to a unique hedging strategy and a unique price. In some other market models,  one needs to deal with BSDEs with non-Lipschitz drivers. Lepeltier and San Martin \cite{LM} discussed existence of solution of such BSDEs and showed the existence of a minimal solution.\\
\indent All of the above references discuss BSDEs driven by Brownian motion or related jump-diffusion process. However, it is known from the work of Kushner \cite{kushner84} that such processes can be approximated by Markov chain models. Hence, there is a motivation for discussing Markov chain model. van der Hoek and Elliott \cite{RE1}  introduced a market model where uncertainties are modeled by a finite state
Markov chain, rather than Brownian motion or related jump diffusions. The
Markov chain has a semimartingale representation involving a vector martingale $M=\{M_t\in\mathbb{R}^N,~t\geq 0\}$. BSDEs in this
framework were introduced by Cohen and Elliott \cite{Sam1} as
$$ Y_t = \xi + \int_t^T f(s, Y_s, Z_s) ds -\int_t^T  Z'_{s}dM_s
,~~~~~t\in[0,T],
$$
where $f$ is Lipschitz in $y$ and $z$. We derived a new comparison theorem in Yang, Ramarimbahoaka and Elliott \cite{YRE} which we think is easier to use in this framework than the Comparison results found in Cohen and Elliott \cite{Sam2} which consider a more general case.\\
\indent In this paper, using the comparison theorem from Yang, Ramarimbahoaka and Elliott \cite{YRE}, we discuss BSDEs in the Markov chain framework with a continuous driver $f$ which has a linear growth in $y$ and Lipschitz in $z$. We follow the method in Lepeltier and San Martin \cite{LM}, that is we construct a monotone sequence of Lipschitz continuous functions of $y$ and $z$ such that they converge to $f$. The existence of solutions will be established followed by the existence of a minimal solution and a corresponding comparison result. An application is given to European option pricing in a market where randomness is modelled by a Markov chain and consumption of investors is taken into account.\\
\indent The present paper is structured as follows: Section 2 will present the model and some preliminary results. Section 3 discusses the existence of multiple solutions of BSDEs with a continuous driver, as well as the minimal solution, followed by the corresponding comparison result. The final section consists of an application to European options.
\section{The  Model and Some Preliminary Results}\label{prelim}
\indent Let $T>0$ and $N\in\mathbb{N}$ be two constants. Consider a finite state Markov chain. Following
 van der Hoek and Elliott \cite{RE1}, we assume the
finite state Markov chain $X=\{X_t, t\geq 0 \}$ is defined on the
probability space $(\Omega,\mathscr{F},P)$ and the state space of
$X$ is identified with the set of unit vectors $\{e_1,e_2\cdots,e_N\}$ in
$\mathbb{R}^N$, where $e_i=(0,\cdots,1\cdots,0) ' $ with 1 in the
$i$-th position. Then the  Markov chain has the semimartingale
representation:
\begin{equation}\label{semimartingale}
X_t=X_0+\int_{0}^{t}A_sX_sds+M_t.
\end{equation}
Here, $A=\{A_t, t\geq 0 \}$ is the rate matrix of the chain $X$ and
$M$ is a vector martingale (See Elliott, Aggoun and Moore
\cite{RE4}).
We assume the elements $A_{ij}(t)$ of $A=\{A_t, t\geq 0 \}$ are bounded for all $t\in[0,T]$. Then the martingale $M$ is square integrable.\\
\indent Take $\mathscr{F}_t=\sigma\{X_s ; 0\leq s \leq t\}$ to be
the $\sigma$-algebra  generated by the Markov process $X=\{X_t\}$
and $\{\mathscr{F}_t\}$ to be its filtration. Since $X$ is right continuous and has left
limits, (written by RCLL), the filtration $\{\mathscr{F}_t\}$ is also
right-continuous. \\
\indent The following is the product rule for semimartingales and we
refer the reader to \cite{elliott} for proof:
\begin{lemma}[Product Rule for Semimartingales]\label{ItoPR}
Let $Y$ and $Z$ be two scalar RCLL semimartingales, with no
continuous martingale part. Then
\begin{equation*}
Y_tZ_t = Y_TZ_T - \int_t^T Y_{s_-} dZ_s - \int_t^T Z_{s_-} dY_s -
\sum_{t < s \leq T} \Delta Z_s \Delta Y_s.
\end{equation*}
Here, $\sum\limits_{0< s \leq t} \Delta Z_s \Delta Y_s$ is the optional
covariation of $Y_t$ and $Z_t$ and is also written as $[Z,Y]_t$.
\end{lemma}
\indent For our (vector) Markov chain $X_t \in \{e_1,\cdots,e_N\}$,
note that $X_t X'_t = \diag(X_t)$. Also, $dX_t= A_t X_tdt+ dM_t$. By Lemma \ref{ItoPR}, we know for $t\in[0,T],$
\begin{align*}
X_tX'_t &= X_0X'_0 + \int_0^t X_{s-} dX'_s + \int_0^t (dX_{s}) X'_{s-} + \sum_{0 < s \leq t} \Delta X_s \Delta X'_s \\
  &= \diag(X_0) + \int_0^t X_s (A_sX_s)' ds + \int_0^t X_{s-} dM'_s + \int_0^t A_s X_s X'_{s-} ds\\
& ~~ + \int_0^t (dM_s ) X'_{s-} + [X,X]_t 
\end{align*}
\begin{align}\label{1}
\nonumber
& = \diag (X_0) + \int_0^t X_s X'_s A'_s ds + \int_0^t X_{s-} dM'_s + \int_0^t A_s X_s X'_{s-} ds \\
&~~ + \int_0^t (dM_s) X'_{s-} + [X,X]_t
- \lel X,X\rir_t + \lel X,X \rir_t.
\end{align}
Recall, $\lel X, X\rir$ is the unique predictable $N\times N$ matrix process such that
$[X,X]-\lel X,X \rir$ is a matrix valued martingale and write
\begin{equation}\label{L_t}
L_t = [X,X]_t - \lel X,X\rir_t, \quad t \in [0,T].
\end{equation}
However,
\begin{equation}\label{2}
X_tX'_t = \diag (X_t) = \diag(X_0) + \int_0^t \diag (A_s X_s)  ds +
\int_0^t \diag(M_s).
\end{equation}
Equating the predictable terms in \eqref{1} and \eqref{2}, we have
\begin{equation}\label{3}
\lel X, X\rir_t =  \int_0^t \diag(A_sX_s) ds - \int_0^t \diag(X_s)
A'_s ds - \int_0^t A_s \diag(X_s) ds.
\end{equation}
Let $\Psi$ be the matrix
\begin{equation}\label{Psi}\Psi_t = \diag(A_tX_t)- \diag(X_t)A'_t - A_t \diag(X_t).
\end{equation}
Then $d\langle X,X\rangle_t=\Psi_tdt.$
For any $t>0$, Cohen and Elliott \cite{Sam1,Sam3}, define the semi-norm $\|.\|_{X_t}$, for
$C, D \in \mathbb{R}^{N\times K}$ as :
\begin{align*}
\lel C, D\rir_{X_t} & = Tr(C' \Psi_tD), \\[2mm]
\|C\|^2_{X_t} & = \lel C, C\rir_{X_t}.
\end{align*}
We only consider the case where $C \in \mathbb{R}^N$, hence we
introduce the semi-norm $\|.\|_{X_t}$ as:
\begin{align}\label{normC}
\nonumber
\lel C, D\rir_{X_t}  = C' \Psi_t D, \\[2mm]
\|C\|^2_{X_t}  = \lel C, C\rir_{X_t}.
\end{align}
It follows from Equation \eqref{3} that
\[\int_t^T \|C\|^2_{X_s} ds = \int_t^T  C' d\lel X, X\rir_s C.\]
For $n \in \mathbb{N}$, denote by $|\cdot|_n$ the Euclidian norm in $\mathbb{R}^n$ and by $\|\cdot\|_{n\times n}$ the norm in $\mathbb{R}^{n \times n}$ such that $\|\Psi\|_{n\times n}= \sqrt{Tr(\Psi' \Psi)}$ for any $\Psi \in \mathbb{R}^{n \times n}$.\\
\indent Lemma \ref{normbound} is Lemma 3.5 in Yang, Ramarimbahoaka and Elliott \cite{YRE}.
\begin{lemma}\label{normbound}
For any $C \in \mathbb{R}^N$,
$$ ~~~~\|C\|_{X_t} \leq \sqrt{3m} |C|_N, ~~\text{ for any }t\in[0,T],$$
where $m>0$ is the bound of $\|A_t\|_{N\times N}$, for any $t\in[0,T]$.
\end{lemma}
\indent Denote by $\mathcal{P}$, the $\sigma$-field generated by the predictable processes defined on $(\Omega, P, \mathcal{F})$ and with respect to the filtration $\{\mathcal{F}_t\}_{t \in [0,\infty)}$. For $t\in[0,\infty)$, consider the following spaces: \\[2mm]
$ L^2(\mathcal{F}_t): =\{\xi;~\xi$ is a $ \mathbb{R} \text{-valued}~ \mathcal{F}_t $-measurable random variable such that $ E[|\xi|^2]< \infty\};$\\[2mm]
$L^2_{\mathcal{F}}(0,t;\mathbb{R}): =\{\phi:[0,t]\times\Omega\rightarrow\mathbb{R};~ \phi$ is an adapted and RCLL process with  $E[\int^t_0|\phi(s)|^2ds]<+\infty\}$;\\[2mm]
$P^2_{\mathcal{F}}(0,t;\mathbb{R}^N): =\{\phi:[0,t]\times\Omega\rightarrow\mathbb{R}^N;~ \phi $ is a predictable process with  $E[\int^t_0\|\phi(s)\|_{X_s}^2ds]<+\infty\}.$\\[2mm]
\indent Lemma \ref{Z2} can be found in Ramarimbahoaka, Yang and Elliott \cite{RYE}.
\begin{lemma}\label{Z2}
For $t\in[0,T]$ and $Z\in P^2_{\mathcal{F}}(0,t;\mathbb{R}^N)$, the following equation holds:
$$
 E [(\int_0^t  Z'_{s} dM_s )^2] = E [ \int_0^t \|Z_s\|^2_{X_s} ds].
$$
\end{lemma}
\indent Lemma \ref{existence} (Theorem 6.2 in Cohen and Elliott \cite{Sam1})
gives the existence and uniqueness result of solutions to BSDEs
driven by Markov chains.
\begin{lemma}\label{existence} Consider the BSDE with Markov chain noise
as follows:
\begin{equation}\label{BSDEMC1}
Y_t = \xi + \int_t^T f(s, Y_s, Z_s ) ds -\int_t^T  Z'_{s} dM_s
,~~~~~t\in[0,T].
\end{equation}
Assume $\xi\in L^2(\mathcal{F}_T)$ and the predictable
function $f: \Omega \times [0, T] \times \mathbb{R} \times
\mathbb{R}^N \rightarrow \mathbb{R}$ satisfies a Lipschitz
condition, in the sense that: there exists two constants $l_1, l_2>0$ such
that for each $y_1,y_2 \in \mathbb{R}$ and $z_1,z_2 \in
\mathbb{R}^{N}$, $t\in[0,T]$,
\begin{equation}\label{Lipch}
|f(t,y_1,z_1) - f(t, y_2, z_2)| \leq l_1 |y_1-y_2| + l_2 \|z_1
-z_2\|_{X_t}.
\end{equation}
We also assume $f$ satisfies
\begin{equation}\label{finite}
 E [ \int_0^T |f^2(t,0,0)| dt] <+\infty.
\end{equation} Then there exists a solution $(Y, Z)\in L^2_{\mathcal{F}}(0,T;\mathbb{R})\times P^2_{\mathcal{F}}(0,T;\mathbb{R}^N)$
to BSDE (\ref{BSDEMC1}). Moreover, this solution is
unique up to indistinguishability for $Y$ and equality $d\langle
X,X\rangle_t$ $\times\mathbb{P}$-a.s. for $Z$.
\end{lemma}
See Campbell and Meyer \cite{campbell} for the following definition:
\begin{definition}[Moore-Penrose pseudoinverse]\label{defMoore}
The Moore-Penrose pseudoinverse of a square matrix $Q$ is the matrix $Q^{\dagger}$ satisfying the properties:\\[2mm]
  1) $QQ^{\dagger}Q = Q$ \\[2mm]
  2) $Q^{\dagger}QQ^{\dagger} = Q^{\dagger}$ \\[2mm]
  3) $(QQ^{\dagger})' = QQ^{\dagger}$ \\[2mm]
  4) $(Q^{\dagger}Q)'=Q^{\dagger}Q.$
\end{definition}
\begin{ass}\label{ass0}
Assume the Lipschitz constant $l_2$ of the driver $f$ given in \eqref{Lipch} satisfies  $$~~~~~l_2\|\Psi_t^{\dagger}\|_{N \times N} \sqrt{6m}\leq 1, ~~~\text{ for any }~t \in [0,T],$$ where $\Psi$ is given in \eqref{Psi} and $m>0$ is the bound of $\|A_t\|_{N\times N}$, for any $t\in[0,T]$.
\end{ass}
\indent The following lemma, which is a comparison result for BSDEs driven by a Markov chain, is found in Yang, Ramarimbahoaka and Elliott \cite{YRE}.\\
\begin{lemma} \label{CT'} For $i=1,2,$ suppose $(Y^{(i)},Z^{(i)})$ is the solution of the
BSDE:
$$Y^{(i)}_t = \xi_i + \int_t^T f_i(s, Y^{(i)}_s, Z^{(i)}_s ) ds
- \int_t^T (Z_{s}^{(i)})' dM_s,\hskip.4cmt\in[0,T].$$
Assume $\xi_1,\xi_2\in L^2(\mathcal{F}_T)$, and $f_1,f_2:\Omega \times [0,T]\times \mathbb{R}\times \mathbb{R}^N \rightarrow \mathbb{R}$ satisfy some conditions such that the above two BSDEs have unique solutions. Moreover assume $f_1$ satisfies \eqref{Lipch} and Assumption \ref{ass0}.
If $\xi_1 \leq \xi_2 $, a.s. and $f_1(t,Y_t^{(2)}, Z_t^{(2)}) \leq f_2(t,Y_t^{(2)}, Z_t^{(2)})$, a.e., a.s., then
$$P( Y_t^{(1)}\leq Y_t^{(2)},~~\text{ for any } t \in [0,T])=1.$$
\end{lemma}
\indent Lemma \ref{LMartin} is proved in Lepeltier and San Martin
\cite{LM} (Lemma 1).
\begin{lemma}\label{LMartin}  Assume $f: \mathbb{R}\rightarrow\mathbb{R}$ is a
continuous function with linear growth, in the sense that there exists a
constant $K\in(0,+\infty)$ such that for any $y\in\mathbb{R},$
$|f(y)|\leq  K(1+|y|).$ Then the sequence of functions
$$
f_n(y)=\inf_{~u \in\mathbb{Q}}\{f(u)+n|y-u|\}
$$
is well defined for $n\geq K$ and satisfies: \\
$ (1)$ linear
growth: for any $ y\in \mathbb{R},~|f_n(y)|\leq K (1+|y|);$\\
$(2)$ monotonicity in
$n:$ for any $ y\in\mathbb{R},~f_n(y)\nearrow;$\\
$(3)$ a Lipschitz continuous condition: for any $
y,u\in\mathbb{R}$, $$|f_n(y)-f_n(u)|\leq
n|y-u|;$$
$(4)$ strong convergence: if $y_n\rightarrow y,~n\rightarrow+\infty$
then $f_n(y_n)\rightarrow f(y),~n\rightarrow+\infty.$
\end{lemma}
\indent Lemma \ref{fatou} can be found in Page 89 in Royden and Fitzpatrick \cite{real} or in Page 172 in Yan and Liu \cite{probability}.
\begin{lemma} \label{fatou} (General Lebesgue Dominated Convergenee Theorem) Let $\{\eta_n\}_{n\in\mathbb{N}}$ and $\{\zeta_n\}_{n\in\mathbb{N}}$ be two sequences of random variables satisfying for any $n\in\mathbb{N}$, $|\eta_n|\leq \zeta_n$ and $\zeta_n$ is integrable. Suppose there exists an integrable random variable $\zeta$ such that $\zeta_n\rightarrow\zeta$, a.e., and $E[\zeta_n]\rightarrow E[\zeta]$. If $\eta_n\rightarrow\eta$, a.e., then
$$E[|\eta_n-\eta|]\rightarrow0,~~~~\mbox{moreover,}~~~~E[\eta_n]\rightarrow E[\eta].$$
\end{lemma}
\section{Existence Theorem of Multiple Solutions to BSDEs in Markov Chains with Continuous Coefficients and a Corresponding Comparison Result}
\indent Consider the following BSDE driven by a Markov chain
\begin{equation}\label{V0}
Y_t= \xi + \int_t^T f(s, Y_s, Z_s) ds - \int_t^T Z'_{s} dM_s
,\hskip.5cmt\in[0,T].
\end{equation}
\begin{theorem}\label{st1} Assume $\xi\in L^2(\mathcal{F}_T)$ and $f: \Omega \times [0, T] \times \mathbb{R}
\times \mathbb{R}^N\rightarrow \mathbb{R}$ is a
$\mathcal{P}\times\mathcal{B}(\mathbb{R}^{1+N})$ measurable function
satisfying\\
(i) linear growth in $y$: there exists a constant $K>0$ such that for each
$\omega\in\Omega,t\in[0,T],y\in \mathbb{R},z\in\mathbb{R}^N$,
$$
|f(w,t,y,z)| \leq K(1+|y|);
$$
(ii) Lipschitz in $z \in \mathbb{R}^N$: there exists a constant $c_2>0$ such that, for any $t \in [0,T]$, $y\in \mathbb{R}$ and $z,z'\in \mathbb{R}^N$: $$|f(w,t,y,z)-f(w,t,y,z')|\leq c_2 \|z-z'\|_{X_t},$$
with $c_2$ satisfying  $$~~~~~~l_2\|\Psi_t^{\dagger}\|_{N \times N} \sqrt{6m}\leq 1, ~~~\text{ for any }~t \in [0,T],$$ where $\Psi$ is given in \eqref{Psi} and $m>0$ is the bound of $\|A_t\|_{N\times N}$, for any $t\in[0,T]$;
(iii) for fixed $(\omega, t)\in\Omega\times[0,T]$, $f(\omega,t,\cdot,\cdot)$ is
continuous.\\[2mm]
 Then there exists a solution
$(Y, Z)\in L^2_{\mathcal{F}}(0,T;\mathbb{R})\times P^2_{\mathcal{F}}(0,T;\mathbb{R}^N)$ of BSDE (\ref{V0}).
\end{theorem}
We follow Lepeltier and San Martin \cite{LM}, who discuss the case of a continuous BSDE driven by Brownian
motion, and proceed with the proof of an existence result for equation
(\ref{V0}).
\begin{proof}
Define for any $n\in\mathbb{N}$, $n\geq K$, $t \in [0,T]$ and $(y,z)\in \mathbb{R}\times \mathbb{R}^N$ the sequence:
\[f_n(t,y,z)= \inf_{u \in \mathbb{Q}} \{f(t,u,z)+ n|y-u|\}.\]
From Lemma \ref{LMartin}, we have for any $t \in [0,T]$, $y,y' \in \mathbb{R}$ and $z,z' \in \mathbb{R}^N$,
\begin{align*}
&\sup_{u\in \mathbb{Q}}\{f(t,u,z)-f(t,u,z')\}-\inf_{u\in \mathbb{Q}}\{f(t,u,z)+n|y'-u|\}\\
&=\sup_{u\in \mathbb{Q}}\{f(t,u,z)-f(t,u,z')\}+\sup_{u\in \mathbb{Q}}\{-f(t,u,z)-n|y'-u|\}\\
&\geq  \sup_{u\in \mathbb{Q}}\{f(t,u,z)-f(t,u,z')-f(t,u,z)-n|y'-u|\}\\
&= \sup_{u\in \mathbb{Q}}\{-f(t,u,z')-n|y'-u|\}\\
&= -\inf_{u\in \mathbb{Q}}\{f(t,u,z')+n|y'-u|\}.
\end{align*}
Then
\begin{align*}
&~f_n(t,y,z) - f_n(t,y',z')\\
& = f_n(t,y,z)-f_n(t,y',z)+f_n(t,y',z)-f_n(t,y',z')\\
& \leq n|y-y'|+\inf_{u\in \mathbb{Q}}\{f(t,u,z)+n|y'-u|\}- \inf_{u\in \mathbb{Q}}\{f(t,u,z')+n|y'-u|\}\\
&\leq n|y-y'|+ \sup_{u\in \mathbb{Q}}\{f(t,u,z)-f(t,u,z')\}\\
& \leq n|y-y'|+ \sup_{u\in \mathbb{Q}} \{c_2 \|z-z'\|_{X_t}\}\\
& =  n|y-y'|+c_2 \|z-z'\|_{X_t}.
\end{align*}
Hence, interchanging the roles of $(y,z)$ and $(y',z')$, we know for any $n\in\mathbb{N}$, $n\geq K$, $t \in [0,T]$, $y,y' \in \mathbb{R}$ and $z,z' \in \mathbb{R}^N$,
$$|f_n(t,y,z) - f_n(t,y',z')|\leq n|y-y'|+c_2 \|z-z'\|_{X_t}.$$
By Lemma \ref{existence} and Lemma \ref{LMartin}, for any $n\in\mathbb{N},$ $n\geq K,$
we deduce that the BSDE
\begin{equation*}Y^{(n)}_t=\xi+\int^T_t f_n(s,Y^{(n)}_s,Z^{(n)}_s)ds
-\int^T_t (Z^{(n)}_{s})'dM_s,~~~t\in[0,T] \nonumber\end{equation*}
has a unique solution $(Y^{(n)},Z^{(n)})\in L^2_{\mathcal{F}}(0,T;\mathbb{R})\times P^2_{\mathcal{F}}(0,T;\mathbb{R}^N)$. So for any $n\in \mathbb{N},n\geq K$, we know $|Y^{(n)}_t|<+\infty$, a.e., a.s. For
$t\in[0,T],y\in\mathbb{R},z\in\mathbb{R}^N$, define
$$\psi(t,y,z)=K(1+|y|).$$ Then for any $t\in[0,T]$,
$\psi(t,y,z)$ is a Lipschitz function in $(y,z)$ and also by Lemma
\ref{existence}, we derive that
the BSDE
\begin{equation*}U_t=\xi+\int^T_t \psi(s,U_s,V_s)ds
-\int^T_t (V_{s})'dM_s,~~~t\in[0,T] \nonumber\end{equation*} has a
unique solution $(U,V)\in L^2_{\mathcal{F}}(0,T;\mathbb{R})\times P^2_{\mathcal{F}}(0,T;\mathbb{R}^N)$. Thus $|U_t|<+\infty,$ a.e., a.s. By Lemma
\ref{LMartin}, we have $f_1\leq f_2\leq\ldots\leq \psi.$ Then by (ii) and
Lemma \ref{CT'}, we have for any $n\in\mathbb{N},n\geq K$, there exists a subset $A_n\subseteq\Omega$ with $P(A_n)=1$ such that for any $\omega\in A_n$, $$Y_t^{(n)}(\omega)\leq Y_t^{(n+1)}(\omega),~~~\text{for any}~ t\in[0,T].$$
Moreover, for any $n\in\mathbb{N},n\geq K$, there exists a subset $B_n\subseteq\Omega$ with $P(B_n)=1$ such that for any $\omega\in B_n$, $$Y_t^{(n)}(\omega)\leq U_t(\omega), ~~~\text{for any}~ t\in[0,T].$$ Hence
\begin{align*}
P(\bigcap_{n=K}^{+\infty}(A_n\cap B_n))& = 1 - P(\bigcup_{n=K}^{+\infty} (A_n^c\cup B^c_n )) \\
&\geq 1 -\sum_{n=K}^{+\infty}(P(A^c_n)+P(B^c_n)) \\
&=1.
\end{align*}
That is,
$$P(Y_t^{(K)}\leq Y_t^{(K+1)}\leq \ldots\leq U_t, ~~\text{for
any }t\in[0,T])=1.$$ For any $\omega\in\Omega,t\in[0,T],$ set
$$Y_t(\omega)=\sup\limits_{n\in\mathbb{N}, n\geq K}Y_t^{(n)}(\omega).$$ Then $$|Y_t|\leq|Y_t^{(K)}|+|U_t|,~ a.e., ~a.s.,$$ hence, $E[\int^T_{0}|Y_t|^2dt]< +\infty$. Moreover, $|Y_t|<+\infty$, a.e., a.s.
 Since $$|Y_t^{(n)}-Y_t|\searrow 0, ~a.e.,~ a.s.$$ when $n\rightarrow+\infty$,
 by Levi's lemma we have
\begin{equation*}
E[\int^T_{0}|Y_t^{(n)}-Y_t|^2dt]\rightarrow0,\hskip0.5cmn\rightarrow+\infty.\end{equation*}
\begin{lemma}\label{estimate}
There exists a constant $C>0$, such that
\begin{equation}\label{sup}
\sup_{n \in
\mathbb{N},~n\geq K}  E [ \int_0^T(|Y^{(n)}_t|^2+
\|Z^{(n)}_t\|^2_{X_t})dt ]\leq C.
\end{equation}
\end{lemma}
\begin{proof}
 By the Stieltjes Chain rule, we known for any $n\in
\mathbb{N},~n\geq K$,
\begin{align*}
|Y^{(n)}_t|^2 = & |\xi|^2 + 2 \int_t^T Y^{(n)}_s f_n(s, Y^{(n)}_s, Z^{(n)}_s) ds   \\
& - 2 \int_t^TY_{s-}^{(n)} (Z^n_s)' dM_s -\sum_{t < s \leq T}
\Delta Y_s^n \Delta Y_s^n.
\end{align*}
Because $\Delta M_s=\Delta X_s$, we have
\begin{align*}
\sum_{t < s \leq T} \Delta Y_s^{(n)} \Delta Y_s^{(n)} & = \sum_{t < s \leq T} (( Z_s^{(n)})' \Delta M_s)((Z_s^{(n)})'\Delta M_s)\\
& = \sum_{t < s \leq T} (Z_s^{(n)})' \Delta X_s \Delta X_s' Z_s^{(n)} \\
& = \int_t^T (Z_s^{(n)})'  (dL_s + d\lel X,X\rir_s) Z_s^{(n)} \\
& = \int_t^T (Z_s^{(n)})' dL_s Z_s^{(n)} + \int_t^T \|Z_s^{(n)}\|_{X_s}^2 ds.
\end{align*}
Let $\beta >0$ be an arbitrary constant.
Using the product rule
for $e^{\beta t}|Y_t^{(n)}|^2$ and from the above equation we derive for any $n\in
\mathbb{N},~n\geq K$,
\begin{align*}
&E [ |Y^{(n)}_0|^2 ]+ E [ \int_0^T \beta |Y^{(n)}_s|^2 e^{\beta s} ds]  + E [ \int_0^T e^{\beta s}\|Z^{(n)}_s\|^2_{X_s} ds] \\
 &= E [ e^{\beta T} |\xi|^2 ] + 2 E [ \int_0^T e^{\beta s} Y^{(n)}_s f_n(s, Y^{(n)}_s , Z^{(n)}_s ) ds ] \\
 &\leq E [e^{\beta T}|\xi|^2] + 2 E[ \int_0^T e^{\beta s}|Y^{(n)}_s | (K(1+ |Y_s^{(n)}|)) ds] 
\end{align*} 
\begin{align*}
 \leq &E [e^{\beta T}|\xi|^2] + 2 E [ \int_0^Te^{\beta s} K |Y^{(n)}_s | ds ] +2 KE [ \int_0^Te^{\beta s} |Y^{(n)}_s |^2ds ]\\
 \leq & E [e^{\beta T}|\xi|^2] + K^2 T e^{\beta T} +E [ \int_0^Te^{\beta s} |Y^{(n)}_s |^2ds ] +2 KE [ \int_0^Te^{\beta s} |Y^{(n)}_s |^2ds ]\\
 \leq & E [e^{\beta T}|\xi|^2] + K^2 T e^{\beta T} + (1+2K)E [ \int_0^Te^{\beta s} |Y^{(n)}_s |^2ds ]
\end{align*}
Set $\beta = 2K+2$. Then, we obtain for any $n\in
\mathbb{N},~n\geq K$,
$$
E [ \int_0^T  |Y^{(n)}_s|^2 e^{\beta s} ds] + E [ \int_0^T e^{\beta s}\|Z^{(n)}_s\|^2_{X_s} ds]
 \leq E [e^{\beta T}|\xi|^2] + K^2Te^{\beta T} .
$$So
\begin{align*}
&E [ \int_0^T  |Y^{(n)}_s|^2 ds] + E [ \int_0^T \|Z^{(n)}_s\|^2_{X_s} ds]\\
&\leq E [ \int_0^T  |Y^{(n)}_s|^2 e^{\beta s} ds]+E [ \int_0^T e^{\beta s}\|Z^{(n)}_s\|^2_{X_s} ds]\\
 &\leq E [e^{\beta T}|\xi|^2] + K^2Te^{\beta T} .
\end{align*}
Since the last line of the above inequality does not depend on $n$, we conclude that there exists a constant $C>0$ such that for all $n \in
\mathbb{N},~n\geq K$, inequality (\ref{sup}) holds.
\end{proof}
We continue the proof of Theorem \ref{st1}. For any $n,p\in
\mathbb{N},~n,p\geq K$, using the product rule for $|Y_t^{(n)} - Y_t^{(p)}|^2$,
\begin{align*}
&~|Y^{(n)}_t- Y_t^{(p)}|^2\\
& =  2 \int_t^T(Y^{(n)}_s-Y_s^{(p)}) (f_n(s, Y^{(n)}_s, Z^{(n)}_s)-f_p(s, Y^{(p)}_s, Z^{(p)}_s)) ds   \\
&-2\int_t^T(Y^{(n)}_{s-}- Y_{s-}^{(p)}) (Z^{(n)}_s- Z_s^{(p)})' dM_s -\sum_{t < s \leq T}
\Delta( Y^{(n)}_s- Y_s^{(p)} )\Delta (Y^{(n)}_s- Y_s^{(p)} ).
\end{align*}
Here,
\begin{align*}
&\sum_{t < s \leq T}\Delta( Y^{(n)}_s- Y_s^{(p)} )\Delta (Y^{(n)}_s- Y_s^{(p)} )\\
& = \sum_{t < s \leq T}((Z^{(n)}_s- Z_s^{(p)})' \Delta M_s) ((Z^{(n)}_s- Z_s^{(p)})' \Delta M_s)  \\
& =\sum_{t < s \leq T} (Z^{(n)}_s- Z_s^{(p)})' \Delta X_s \Delta X_s'(Z^{(n)}_s- Z_s^{(p)})
\end{align*}
\begin{align*}
=& \int_t^T (Z^{(n)}_s- Z_s^{(p)})' (dL_s + d\lel X,X\rir_s) (Z^{(n)}_s- Z_s^{(p)}) \\
= & \int_t^T  (Z^{(n)}_s- Z_s^{(p)})'dL_s  (Z^{(n)}_s- Z_s^{(p)}) + \int_t^T \| Z^{(n)}_s- Z_s^{(p)}\|_{X_s}^2 ds.
\end{align*}
Set $t=0$, taking the expectation on both sides of the above equation, we deduce
\begin{align*}
&E[| Y^{(n)}_0- Y_0^{(p)}|^2] + E [ \int_0^T \|Z^{(n)}_s- Z_s^{(p)}\|_{X_s}^2 ds]\\
&=2 E [ \int_0^T (Y^{(n)}_s-Y_s^{(p)}) (f_n(s, Y^{(n)}_s, Z^{(n)}_s)-f_p(s, Y^{(p)}_s, Z^{(p)}_s)) ds].
\end{align*}
So by Lemma \ref{estimate} we know there exists a constant $C' >0$ depending on the constant $C$ given in Lemma \ref{estimate} such that
\begin{align*}
& E [ \int_0^T \|Z^{(n)}_s- Z_s^{(p)}\|_{X_s}^2 ds]\\
&\leq2 E [ \int_0^T (Y^{(n)}_s-Y_s^{(p)}) (f_n(s, Y^{(n)}_s, Z^{(n)}_s)-f_p(s, Y^{(p)}_s, Z^{(p)}_s)) ds]\\
&\leq 2 (E [ \int_0^T |Y^{(n)}_s-Y_s^{(p)}|^2ds])^\frac{1}{2} (E[\int^T_0|f_n(s, Y^{(n)}_s, Z^{(n)}_s)-f_p(s, Y^{(p)}_s, Z^{(p)}_s)|^2 ds])^\frac{1}{2}\\
&\leq 2 (E [ \int_0^T |Y^{(n)}_s-Y_s^{(p)}|^2ds])^\frac{1}{2} (E[\int^T_0(|f_n(s, Y^{(n)}_s, Z^{(n)}_s)|+|f_p(s, Y^{(p)}_s, Z^{(p)}_s)|)^2 ds])^\frac{1}{2}\\
&\leq 2 (E [ \int_0^T |Y^{(n)}_s-Y_s^{(p)}|^2ds])^\frac{1}{2} K(E[\int^T_0(2+|Y^{(n)}_s|+|Y^{(p)}_s|)^2 ds])^\frac{1}{2}\\
&\leq2 (E [ \int_0^T |Y^{(n)}_s-Y_s^{(p)}|^2ds])^\frac{1}{2} K(3E[\int^T_0(4+|Y^{(n)}_s|^2+|Y^{(p)}_s|^2) ds])^\frac{1}{2}\\
&\leq 2KC' (E [ \int_0^T |Y^{(n)}_s-Y_s^{(p)}|^2ds])^\frac{1}{2}.
\end{align*}
Hence, $\{Z^{(n)},n\in\mathbb{N}, n\geq K\}$ is a Cauchy sequence in $P^2_{\mathcal{F}}(0,T;\mathbb{R}^N)$. Consider the factor space of equivalence classes of processes in $P^2_{\mathcal{F}}(0,T;\mathbb{R}^N)$. An equivalence class is just all processes which differ by a null process. On that space the semi norm is actually a norm and so the space is complete. Then there exists a process $Z\in P^2_{\mathcal{F}}(0,T;\mathbb{R}^N)$ such that \begin{equation}\label{znz}E [ \int_0^T \|Z^{(n)}_t- Z_t\|_{X_t}^2 dt]\rightarrow0,~~~~n\rightarrow+\infty.\end{equation}
Also,
\begin{align*}
&~|f_n(t,Y_t^n,Z_t^n)-f(t,Y_t,Z_t)|\\
&\leq |f_n(t,Y_t^n,Z_t^n)-f_n(t,Y^n_t,Z_t)| + |f_n(t,Y_t^n,Z_t)-f(t,Y_t,Z_t)| \\
&\leq c_2 \|Z_t^n-Z_t\|_{X_t} + |f_n(t,Y_t^n,Z_t)-f(t,Y_t,Z_t)|.
\end{align*}
Thus by \eqref{znz} and Lemma \ref{LMartin} (4),
we have $$f_n(t,Y^{(n)}_t,Z^{(n)}_t)\rightarrow f(t,Y_t,Z_t),~~~{a.e., a.s.}$$
Let $Q$ be a probability on $\Omega\times[0,T]$ satisfying $Q|_{\Omega}=P$ and $dQ|_{[0,T]}=\dfrac{dt}{T}$. Denote the expectation under $Q$ by $E^Q[\cdot]$. Thus for any $n\in \mathbb{N}$, $n\geq K$, we obtain
$$E^Q[K(1+|Y^{(n)}_t|)]<+\infty.$$
Moreover, when $n\rightarrow+\infty$, we derive
$$\begin{array}{lll}
(1)~f_n(t,Y^{(n)}_t,Z^{(n)}_t)\rightarrow f(t,Y_t,Z_t),~~~\mbox{Q-a.e.;}\\[3mm]
(2)~K(1+|Y^{(n)}_t|)\rightarrow K(1+|Y_t|),\hskip.5cm\mbox{Q-a.e.;}\\[3mm]
(3)~E^Q[K(1+|Y^{(n)}_t|)]\rightarrow E^Q[K(1+|Y_t|)]<+\infty.
\end{array}$$
Since for any $n\in \mathbb{N}$, $n\geq K$, $t\in[0,T]$,
$$f_n(t,Y^{(n)}_t,Z^{(n)}_t)\leq K(1+|Y^{(n)}_t|),$$
we have by Lemma \ref{fatou}, when $n\rightarrow+\infty$,
$$E^Q[|f_n(t,Y^{(n)}_t,Z^{(n)}_t)-f(t,Y_t,Z_t)|]\rightarrow0,$$
hence,
$$E[\int^T_0|f_n(t,Y^{(n)}_t,Z^{(n)}_t)-f(t,Y_t,Z_t)|dt]\rightarrow0.$$
Therefore, we conclude when $n\rightarrow+\infty$,
\begin{align*}
&~E[\sup\limits_{t\in[0,T]}|\int^T_tf_n(s,Y^{(n)}_s,
Z^{(n)}_s)ds-\int^T_t f(s,Y_s,Z_s)ds|]\\
&\leq E[\sup\limits_{t\in[0,T]}\int^T_t|f_n(s,Y^{(n)}_s
Z^{(n)}_s)- f(s,Y_s,Z_s)|ds]\\
& =E[\int^T_0|f_n(s,Y^{(n)}_s,Z^{(n)}_s)-f(s,Y_s,Z_s)|ds]\rightarrow 0.
\end{align*}
By Doob's martingale inequality and Lemma \ref{Z2}, we know when $n\rightarrow+\infty,$
$$\begin{array}{lll}
E [\sup\limits_{t\in[0,T]}|\int_t^T(Z^{(n)}_s- Z_{s})' dM_s |^2]\\[3mm]
=E [\sup\limits_{t\in[0,T]}|\int_0^T(Z^{(n)}_s- Z_{s})' dM_s -\int_0^t(Z^{(n)}_s- Z_{s})' dM_s |^2]\\[3mm]
\leq 2E [|\int_0^T(Z^{(n)}_s- Z_{s})' dM_s|^2 +\sup\limits_{t\in[0,T]}|\int_0^t(Z^{(n)}_s- Z_{s})' dM_s |^2]\\[3mm]
\leq10E [|\int_0^T(Z^{(n)}_s- Z_{s})' dM_s |^2]=10 E [ \int_0^T \|Z^{(n)}_s- Z_{s}\|^2_{X_s} ds]\rightarrow0.\end{array}$$
So $(Y,Z)$ satisfies BSDE (\ref{V0}).
\end{proof}
\begin{theorem}\label{minimal}
We make the same assumptions as in Theorem \ref{st1}.
Then there is a minimal
solution $\bar{Y}$ of (\ref{V0})$,$ in the sense that for
any other solution $Y$ of (\ref{V0})$,$ we have
\[P (\bar{Y}_t \leq Y_t, \text{ for any } t \in [0,T])=1.\]
\end{theorem}
\begin{proof}
By Theorem \ref{st1}, there is a solution
$(Y',Z')\in  L^2_{\mathcal{F}}(0,T;\mathbb{R})\times P^2_{\mathcal{F}}(0,T;\mathbb{R}^N)$ of BSDE (\ref{V0}). By Lemma \ref{CT'}, we have for any $n\in\mathbb{N}, n\geq K$,
 \[P(Y_t^{(n)}\leq Y'_t, \text{ for any } t\in[0,T])=1,\]
here $Y_t^{(n)}$ is the same as in the proof of Theorem \ref{st1}. That is, for any $n \in \mathbb{N}, n\geq K$, there is a subset $F_n \subseteq \Omega$ such that $P(F_n)=1$ and for all $\omega \in F_n$, $Y^{(n)}_t \leq Y'_t$, for any $t \in [0,T]$. Thus,
$$\begin{array}{lll}P(Y_t\leq Y'_t,~\text{ for any }\ t \in [0,T])\\[2mm]=P( \sup\limits_{n\in\mathbb{N}, n\geq K}Y^{(n)}_t\leq Y'_t,~ \text{ for any }t \in [0,T])\\[2mm]
=P(\bigcap\limits^{+\infty}_{n=K}F_n)=1-P(\bigcup\limits^{+\infty}_{n=K}F_n^c)\\[2mm]\geq 1-\sum\limits^{+\infty}_{n=K}P(F_n^c)=1,\end{array}$$ 
that is, $Y$ is the minimal solution.
\end{proof}
\indent Since the solutions of BSDEs with Markov chain noise and continuous coefficients are not unique, we cannot give comparison results for all solutions. However, noticing the minimal solution is unique for a BSDE of this kind we can compare the minimal solutions of these BSDEs.\\
\indent Consider the following two BSDEs for Markov chain noise:
\begin{equation}
Y_t = \xi_1 + \int_t^T f(s, Y_s, Z_s ) ds -\int_t^T  Z'_{s} dM_s
,~~~~~t\in[0,T]
\nonumber\end{equation}
and
\begin{equation}U_t=\xi_2+\int^T_t g(s,U_s,V_s)ds
-\int^T_t V_{s}'dM_s,~~~t\in[0,T]. \nonumber\end{equation}
\begin{theorem}\label{cct}  Assume $\xi_1,\xi_2 \in L^2(\mathcal{F}_T)$, $f$ and $g$ both satisfy all the conditions of Theorem \ref{st1}. Denote the minimal solutions of the above two BSDEs by $\bar{Y}$ and $\bar{U}$, respectively. If $\xi_1 \leq \xi_2$,
and for any $t\in[0,T],y\in\mathbb{R},z\in\mathbb{R}^N$, $ f(t,y,z) \leq g(t,y,z),$
 then
$$P(\bar{Y}_t\leq \bar{U}_t,~\mbox{for~all}~~t\in[0,T])=1.$$\end{theorem}
\begin{proof} Similarly to the proof of Theorem \ref{st1}, there exists a constant $K'\in(0,+\infty)$ such that we can denote for fixed $(t,\omega)\in [0,T]\times\Omega$, the sequence associated with $f(t,y,z)$ by $f_n(t,y,z),~n\in\mathbb{N},n\geq K',$ the sequence associated with $g(t,y,z)$ by $g_n(t,y,z),~n\in\mathbb{N},n\geq K'$. Then we have for any $n\in\mathbb{N},n\geq K', t\in[0,T],y\in\mathbb{R},z\in\mathbb{R}^N$, $ f_n(t,y,z) \leq g_n(t,y,z)$. By Lemma \ref{existence}, for any $n\in\mathbb{N},$ $n\geq K',$
we deduce that the BSDE
\begin{equation}Y^{(n)}_t=\xi_1+\int^T_t f_n(s,Y^{(n)}_s,Z^{(n)}_s)ds
-\int^T_t (Z^{(n)}_{s})'dM_s,~~~t\in[0,T] \nonumber\end{equation}
has a unique solution $(Y^{(n)},Z^{(n)})\in L^2_{\mathcal{F}}(0,T;\mathbb{R})\times P^2_{\mathcal{F}}(0,T;\mathbb{R}^N)$ and the BSDE
\begin{equation}U^{(n)}_t=\xi_2+\int^T_t g_n(s,U^{(n)}_s,V^{(n)}_s)ds
-\int^T_t (V^{(n)}_{s})'dM_s,~~~t\in[0,T] \nonumber\end{equation}
has a unique solution $(U^{(n)},V^{(n)})\in L^2_{\mathcal{F}}(0,T;\mathbb{R})\times P^2_{\mathcal{F}}(0,T;\mathbb{R}^N)$. By Lemma \ref{CT'}, we obtain for any $n\in\mathbb{N},$ $n\geq K',$ there exists there exists a subset $A_n\subseteq\Omega$ with $P(A_n)=1$ such that for any $\omega\in A_n$,
$$Y_t^{(n)}(\omega)\leq U_t^{(n)}(\omega),~~~~\mbox{for~all}~~t\in[0,T].$$
Similarly to the proof of Theorem \ref{st1} we know $P(\bigcap_{n=K}^{+\infty}A_n)=1$. That is,
$$P(Y_t^{(n)}\leq U_t^{(n)},~~~~\mbox{for~all}~~n\in\mathbb{N},n\geq K',t\in[0,T])=1.$$
So for a.e. $\omega\in\Omega,$
\begin{align*}
\bar{Y}_t=\sup\limits_{n\in\mathbb{N}, n\geq K'}Y^{(n)}_t\leq \sup\limits_{n\in\mathbb{N}, n\geq K'}U^{(n)}_t
=\bar{U}_t,~~\mbox{for~all}~~t\in[0,T].
\end{align*}
\end{proof}
\section{Application to European Options}
 \indent It is shown in \cite{KQ2}, for a market where the underlying securities follow a geometric Brownian motion model, that the pricing of a European option can be formulated in terms of BSDEs driven by a Brownian motion.\\
\indent  In this section, $T$ will be the time horizon. We consider a market composed of a bond $S^0$, whose price dynamics are
\begin{equation*}
dS^0_t = r_t S^0_t dt, ~~~~~t\in [0,T],
\end{equation*}
and $N$ stocks $S^i$, $i=1,\cdots,N$ whose price dynamics are
\begin{equation*}
dS^i_t = S^i_{t-}(g_t^i dt + \sum_{j=1}^N h_{t}^{ij} dM^j_t),~~~~~t\in [0,T].
\end{equation*}
Here, at any time $t \in [0,T]$, $r_t$ is the interest rate , $g_t^i\in \mathbb{R}$ is the appreciation rate of the stock $S^i$ and $h_t= (h_t^{ij}) \in \mathbb{R}^{N \times N}$ is the volatility matrix.
We assume that:
\begin{enumerate}
\item The interest rate $r$ is a non-negative predictable process.
\item The appreciation rate $g$ is a predictable process in $\mathbb{R}^N$.
\item The volatility $h$ is also a predictable process in $\mathbb{R}^{N\times N}$ and is invertible.
\end{enumerate}
For $i=1,\cdots,N$, write $\bar{S}^i_t = e^{-\int_0^t r_s ds} S^i_t$ for the discounted stock price at time $t \in [0,T]$.
\subsection*{No-arbitrage assumption}
\indent Recall from asset pricing theory  that the existence of an equivalent martingale measure ensures no-arbitrage. That is, we need to find a measure $Q$ equivalent to $P$ under which, for each $i$, the discounted price $\bar{S}^i$ is a martingale.\\
\indent The following lemmas are from \cite{elliott}.
\begin{lemma}\label{Dolean}
Suppose $\{Y_t\}_{t \geq 0}$, is a semimartingale and suppose $X_{0-}=0$ a.s. Then there is a unique semimartingale $\{Z_t\}$ such that  \[Z_t = Z_{0-} + \int_0^t Z_{s-}dX_s.\] Furthermore, $Z_t$ is given by the expression
\[Z_t = Z_{0-} \exp (X_t - \frac{1}{2}\lel X^c, X^c\rir_t) \prod_{0 \leq s \leq t} (1+ \Delta X_s)e^{-\Delta X_s},\]
for $t \geq 0$, where the infinite product is absolutely convergent almost surely.
\end{lemma}
\begin{lemma}\label{GT}(Girsanov Transformation)
 Let $Q$ be an equivalent measure to $P$
and put $N_t$ as the RCLL version of $\{E \left[ \frac{dQ}{dP}|\mathcal{F}_t\right], t \geq 0\}$. If $\hat{M}$ is a $P$-local martingale such that $\hat{M}_0=0$ then
$- \int_0^t \frac{1}{N} d[N, \hat{M}] + \hat{M}$ is a local martingale under the measure $Q$.
\end{lemma}
\indent In our discussion, we assume that there exists a predictable process $\theta_t \in \mathbb{R}^N$ such that:
\begin{equation}\label{teta}
g_t- r_t \textbf{1} = h_t \theta_t,
\end{equation}
with $|\theta| \leq K_0$ for some constant $K_0 >0$.
For each $i =1,\cdots,N$, applying It$\hat{\text{o}}$'s product rule to $e^{-\int_0^t r_s ds} S^i_t$, noting $[e^{- \int_0^t r_s ds}, S^i]  = 0$ and using \eqref{teta}, we obtain for $i=1,\cdots,N$,
\begin{align*}
 d\bar{S}^i_t  &= d e^{- \int_0^t r_s ds} S^i_t = -r_t e^{- \int_0^t r_s ds} S^i_t dt + e^{- \int_0^t r_s ds}  dS^i_t \\
& =  -r_t e^{- \int_0^t r_s ds} S^i_t dt  + e^{- \int_0^t r_s ds} S^i_t g_t^i dt + e^{- \int_0^t r_s ds} S^i_{t-} \sum_{j=1}^N h_{t}^{ij}dM^j_t \\
& = \bar{S}_{t-}^i ((-r_t + g_t^i)dt + \sum_{j=1}^N h_{t}^{ij}dM^j_t ) \\
& = \bar{S}_{t-}^i ((h_t\theta_t)^i dt+ \sum_{j=1}^N h_{t}^{ij}dM^j_t ) \\
&= \bar{S}^i_{t-} (\sum_{j=1}^N h_{t}^{ij}(\theta^j_t dt + dM^j_t)).
\end{align*}
Let $Y_t$ be the vector process satisfying
\begin{equation}\label{Y}
dY_t = h_{t}(\theta_t dt + dM_t).
\end{equation}
Note, $Y$ is a semimartingale and from Lemma \ref{Dolean}, the unique solution to
\[d\bar{S}^i_t = \bar{S}^i_{t-} d Y^i_t, ~ i=1,\cdots,N,\]
is the stochastic exponential $\bar{S}^i_t = \bar{S}^i_0 \mathcal{E}(Y^i)_t$, where
\begin{equation*}
\mathcal{E}(Y^i)_t = \exp(Y^i_t - \frac{1}{2} \lel (Y^i)^c,(Y^i)^c\rir _t) \prod\limits_{0 \leq s \leq t} (1+ \Delta Y^i_s)e^{- \Delta Y^i_s}.
\end{equation*}
\begin{lemma}\label{qpmart}
If $Y$ given in \eqref{Y} is a martingale under some measure $Q$ equivalent to $P$ then $\mathcal{E}(Y)$ is also a  martingale under $Q$.
\end{lemma}
\begin{proof}
The proof is straightforward since for $i=1,\cdots,N$, $\bar{S}^i_t = \bar{S}^i_0 + \int_0^t \bar{S}^i_s dY^i_s$ and $\int_0^t \bar{S}^i_s dY^i_s$ is a martingale, hence $\bar{S}^i_t$ is a martingale, so is $\bar{S}_t$ therefore $\mathcal{E}(Y)$ is also a martingale.
\end{proof}
\indent Now, write $\hat{M}_t = \int_0^t h_s dM_s$. Then $\hat{M}$ is a martingale with $\hat{M}_0 =0$. Suppose there is a uniformly integrable martingale  process $N$ satisfying
\[- \frac{1}{N_t} d[N,\hat{M}]_t = h_t \theta_t dt, ~~ \text{for any}~ t \in [0,T],\] and an equivalent measure $Q$ such that $N_t =E \left[ \frac{dQ}{dP}|\mathcal{F}_t\right]$, then by Lemma \ref{GT},
$Y$ in \eqref{Y} is a martingale under $Q$ and so is $\bar{S}$, by Lemma \ref{qpmart}.
$Q$ is then an equivalent martingale measure for the market  which ensures there is no-arbitrage opportunity.\\
\indent Now, assume that investors consume continuously a part of their wealth or profit. A consumption rate, at time $t$, for an investor is denoted by $c_t$ which is an adapted process and the cumulative spending is $C_t = \int_0^t c_s ds$. In our discussion, we consider investors who decide to limit their consumption, that is a positive constant amount $K_1$ is chosen by each investor such that $|c_t|< K_1$. \\
\indent Given the consumption rate model, we shall find a strategy which replicates the European option at the exercise time $T < \infty$. In other words, we shall determine the amount of money that we shall invest in the above securities in order to be able to pay off the option at maturity time $T$. Therefore, here, a strategy is a couple $(V, (\pi^0, \pi))$ where $V \in \mathbb{R}$ is the portfolio value, $\pi^0_t$ is number of bonds held at time $t$  and $\pi_t=(\pi_t^1, \cdots,\pi_t^N)$ such that for $i=1, \cdots,N$, $\pi_t^i$ is the number of stocks $i$ held at time $t$. The portfolio value $V$ of the investor at any time $t \in [0,T]$ is then $V_t = \sum_{i=0}^N \pi^i_t S^i_t$. The strategy needs to be self-financing, that is any increase and decrease in his wealth $V$ comes from gains and losses from the investment and the consumption. Such  a strategy satisfies
\begin{equation*}
dV_t = \pi^0_t dS^0_t + \sum_{i=1}^N \pi^i_{t-} dS^i_t - dC_t.
\end{equation*}
Using the dynamics of $S^0$, $S^i$, $i=1,\cdots,N$ and \eqref{teta}, we have:
\begin{align*}
dV_t & = \pi_t^0 r_t S_t^0 dt + \sum_{i=1}^N S_t^i\pi_t^i g_t^i dt + \sum_{i=1}^N \pi_{t-}^i S_{t-}^i \sum_{j=1}^N h_{t}^{ij} dM^j_t -dC_t \\
& =  r_t (V_t - \sum_{i=1}^N \pi_t^i S_t^i) dt + \sum_{i=1}^N S_t^i\pi_t^i g_t^i dt + \sum_{i=1}^N \pi_t^i S_{t-}^i \sum_{j=1}^N h_{t}^{ij} dM^j_t -dC_t 
\end{align*}
\begin{align*}
& = r_tV_t dt + \sum_{i=1}^N \pi_t^i S_t^i (-r_t + g_t^i) dt + \sum_{i=1}^N \pi_{t-}^i S_{t-}^i \sum_{j=1}^N h_{t}^{ij} dM^j_t -dC_t \\
& = r_tV_t dt + \sum_{i=1}^N \pi_t^i S_t^i \sum_{j=1}^N h_t^{ij} \theta_t^j dt + \sum_{i=1}^N \pi_{t-}^i S_{t-}^i \sum_{j=1}^N h_{t}^{ij} dM^j_t -dC_t.
\end{align*}
That is,
\[dV_t = r_t V_t + (\diag (S_t)\pi_t)' h_t \theta_t dt + (\diag (S_t)\pi_t)' h_t dM_t- dC_t.\]
Writing the backward integral form of the above, for $t \in [0,T]$, we have:
\begin{equation}\label{V1}
V_t = V_T + \int_t^T (c_s - r_sV_s - (\diag (S_s)\pi_s)' h_s \theta_s) ds - \int_t^T (\diag (S_{s-})\pi_{s-})' h_{s} dM_s.
\end{equation}
Now, we have the following definitions:
\begin{definition}
A hedging strategy for a European option whose payoff at time $T$ is $\xi$, is a self-financing strategy $(V,\pi)$ such that $V_T = \xi$ with $E[ \int_0^T | h_s' \pi_s|_N^2 ds] < \infty.$ If such a strategy exists, the European option is called hedgeable.
\end{definition}
Denote the set of strategies given in the above definition by $\mathcal{S}(\xi)$ and the {\it fair price} at time $t$ of the European option by $P_t$. We use the following definition from  \cite{karatzasI}.
\begin{definition}\label{fp}
The fair price $P_t$ at time $t$ of the hedgeable option is the smallest amount needed to hedge the option. That is
\[P_t = \inf \{ x \geq 0;~ \text{there exists}~ (V, \pi) \in \mathcal{S}(\xi) \text{ such that } V_t=x\}.\]
\end{definition}
Let 
\begin{equation}\label{ff}f(t,v,z) = c_t -r_t v - z'\theta_ t.
\end{equation}
For some constant $K_2 >0 $,  we assume $|r_t|\leq K_2$. Take 
$$K_3= \max\{K_1,K_2,K_0\sqrt{3m}\}.$$ From Lemma \ref{normbound}, for any $t \in [0,T]$
\begin{align*}
|f(t, v, z)|& \leq |c_t|+|r_t||v|+|\theta'_t z|\\
& \leq K_1 + K_2 |v| + K_0 \sqrt{3m}\|z\|_{X_t} \\
& \leq K_3 (1 + |v| + \|z\|_{X_t}).
\end{align*}
Since investors can only hold a finite number of shares, it is reasonable to suppose that $z$ is bounded.
 Therefore, there is a constant $K'_3$ such that
\[|f(t, v, z)| \leq K'_3(1+|v|).\]
Hence $f$ is linear increasing in $y$ with constant $K'_3$. Also $f$ is Lipschitz in $z$ with constant $K_3$. Consequently, we have the following proposition:
\begin{proposition}
Assume $f$ in equation \eqref{ff} satisfies $K'_3\|\Psi_t^{\dagger}\|_{N\times N} \sqrt{6m} \leq 1$. Let $\xi\in L^2(\mathcal{F}_T)$. Then $\mathcal{S}(\xi)$ is non-empty.
\end{proposition}
\begin{proof}
By Theorem \ref{st1}, a solution $(V,Z)$ to the BSDE
\begin{equation}\label{V2}
V_t = \xi + \int_t^T (c_s - r_sV_s - Z'_s \theta_s) ds - \int_t^T Z'_{s} dM_s
\end{equation}
exists. Let $(V,Z)$ be such a solution, then a strategy satisfying \eqref{V1} exists if there is a solution $\pi_t$ to the equation $Z_t =h'_t\diag (S_t)\pi_t$, for any $t \in [0,T]$. Since $h$ and $\diag (S_t)$ are invertible, the equation admits a unique solution $\pi_t$.
Therefore, the set $\mathcal{S}(\xi)$ of hedging strategies is non-empty.
\end{proof}
\begin{proposition}
Assume $f$ in equation \eqref{ff} satisfies $K'_3\|\Psi_t^{\dagger}\|_{N\times N} \sqrt{6m} \leq 1$. The fair price $P_t$, at time $t \in [0,T]$ of the European option exists and satisfies:
\begin{equation*}
P_t = E \left[ \xi + \int_t^T (c_s - r_s V_s - Z'_s \theta_s) ds| \mathcal{F}_t\right],
\end{equation*}
where $(V,Z)$ is one pair solution of \eqref{V2}. Moreover, $V$ is minimal.
\end{proposition}
\begin{proof}
Theorem \ref{minimal} ensures there is a minimal solution of \eqref{V2}. Let $V$ be such a minimal solution. Then from Definition \ref{fp}, $P_t = V_t$ is the fair price of the European option at any time $t \in [0,T]$. Taking the expectation in  \eqref{V2}, given the information at time $t$, we obtain the result.
\end{proof}
\begin{proposition}
Let $(V^{(1)}, Z^{(1)})$ and $(V^{(2)}, Z^{(2)})$, be two solutions of \eqref{V2}.\\
 Write $\bar{V}^{(i)}_t =e^{-\int_0^t r_s ds} V^{(i)}_t$, $i=1,2$. Then for  $t \in [0,T]$,\[E^{Q}\left[ \bar{V}^{(1)}_t - \bar{V}^{(2)}_t \right]=0\]
and
\[E^{Q} [\bar{V}^{(1)}_t ] = E^{Q} [\bar{V}^{(2)}_t ] =E^{Q}\left[\xi + \int_t^T e^{- \int_0^s r_u du} c_s ds \right].\]
\end{proposition}
\begin{proof}
For $i=1,2$, let $\pi^{(i)}$ be the solution of $Z_t^{(i)}=h'_t \diag(S_t)\pi_t^{(i)}$, for any $t \in [0,T]$. Using the product rule on $ e^{-\int_0^t r_s ds} V^{(i)}_t$ and using $Y$ in \eqref{Y}, we have
\begin{align*}
d\bar{V}_t^{(i)} & = - e^{- \int_0^t r_s ds} c_t dt + e^{- \int_0^t r_s ds} (Z_{t}^{(i)})'(\theta_t dt + dM_t)\\
& = - e^{- \int_0^t r_s ds} c_t dt + e^{- \int_0^t r_s ds} (\diag(S_t)\pi_t^{(i)})'h_t (\theta_t dt + dM_t)\\
&= - e^{- \int_0^t r_s ds} c_t dt + e^{- \int_0^t r_s ds} (\diag(S_t)\pi_t^{(i)})'dY_t.
\end{align*}
Since $Y$ is martingale under $Q$, for $t \in [0,T]$, integrating the above from $t$ to $T$ and taking the expectation, for $i=1,2$, we derive
\[ E^{Q} [\bar{V}_t^{(i)}] = E^{Q}\left[\xi+ \int_t^T e^{- \int_0^s r_u du} c_s ds \right].\]
This proves that under the measure $Q$, for all solutions $(V,Z)$ of \eqref{V2}, all the $V$'s are equal, hence the price is unique under the risk neutral measure $Q$.
\end{proof}
\section*{Conclusion}
The paper discusses backward stochastic differential equations with Markov chain noise. Existence is established when the driver is not Lipschitz. The minimal solution is shown to be unique. The result is applied to pricing European options in a Markov chain market.


\begin{thebibliography}{11}

\bibitem{campbell} L. Campbell and D. Meyer, Generalized inverses of linear transformations, SIAM, (2008).

\bibitem{Sam1}
Cohen, S. N. and Elliott, R. J. (2008) Solutions of backward stochastic differential equations in Markov
chains,{\itshape
  Communications on Stochastic Analysis}, 2(2), pp.  251--262.

\bibitem{Sam2}
Cohen, S. N. and Elliott, R. J. (2010) Comparison theorems for finite state
backward stochastic differential equations, {\itshape Contemporary
Quantitative Finance} (Springer).

\bibitem{Sam3}
S. N. Cohen and R. J. Elliott,
Comparisons for Backward Stochastic Differential Equations on Markov Chains and Relate No-arbitrage Conditions, Annals of Applied Probability, \textbf{20}(1), 267-311 (2010).

\bibitem{KQ}
El Karoui, N. and Quenez, M.C. (1997) Non-linear pricing theory and backward stochastic differential equations, {\itshape
Lecture Notes in Mathematics}, 1657, pp. 191--246.

\bibitem{KQ2}
El Karoui, N. and Quenez, M.C. (1997) Imperfect markets and backward stochastic differential equations, {\itshape Numerical Methods in Finance}, pp. 181--214.

\bibitem{elliott}
Elliott, R.~J. (1982) {\itshape Stochastic Calculus and Applications}, Stochastic Calculus and Applications,
{Springer-Verlag, New York Heidelberg Berlin 1982}.

\bibitem{RE4}
Elliott, R.~J., Aggoun, L. and Moore,J. B. (1994) {\itshape Hidden Markov Models: Estimation and Control,
  Applications of Mathematics}, 29 (Springer-Verlag, Berlin-Heidelberg-New York).

\bibitem{karatzasI}
Karatzas, I. and Shreve, S. (1987) {\itshape Brownian Motion and Stochastic Calculus} (New York: Springer Verlag).

\bibitem{kushner84}
Kushner, H. J. (1984) {\itshape Approximation and Weak Convergence Methods for Random Processes with Applications to Stochastic Systems Theory}. (MIT Press, Cambridge, Mass.)

\bibitem{LM}
Lepeltier, J.~P. and San Martin, J. (1997) Backward stochastic differential equations with
continuous coefficient, {\itshape Statistics and Probability Letters}, 32(4), pp. 425--430.

\bibitem{ParPeng1}
Pardoux, E. and Peng, S. (1990) Adapted solution of a backward stochastic differential equation, {\itshape Systems and Controls Letters}, 14,
pp. 55--61.

\bibitem{RYE}
Ramarimbahoaka, D., Yang, Z. and Elliott, R.~J. (2014) Reflected Backward Stochastic Differential Equations for a Finite State Markov Chain Model and American Options, arXiv:1404.2218.


\bibitem{real}
Royden, H.~L. and Fitzpatrick, P. M. (1993) {\itshape Real Analysis (4th Edition)} (Macmillan Publishing Company).

\bibitem{RE1}
van der Hoek, J. and Elliott R.~J. (2010) {\itshape Asset Pricing Using Finite State Markov Chain Stochastic Discount Functions}, Stochastic Analysis and Applications.

\bibitem{probability}
Yan, S.~J. and Liu, X.~F. (1994) {\itshape Measure and Probability (Chinese Edition)} ( Beijing Normal University press).

\bibitem{YRE}
Z. Yang, D. Ramarimbahoaka and R. J. Elliott (2014), Comparison and converse comparison theorems for backward stochastic differential equations with {M}arkov chain Noise, arxiv.org/abs/1404.2213.

\end{thebibliography}
\end{document}